\documentclass[a4paper,11pt]{amsart}

\usepackage{amsmath,amssymb,amsthm,latexsym,amscd,mathrsfs}
\usepackage{indentfirst}
\usepackage{stmaryrd}
\usepackage{graphicx}
\usepackage{extarrows}
\usepackage{multirow}
\usepackage{latexsym}
\usepackage{amsfonts}
\usepackage{color}
\usepackage{pictexwd,dcpic}
\usepackage{graphicx}
\usepackage{psfrag}
\usepackage{hyperref}
\usepackage{comment}
\usepackage{enumerate}

\numberwithin{equation}{section} \theoremstyle{plain}
\newtheorem{thm}{Theorem}[section]

\newtheorem{prop}[thm]{Proposition}
\newtheorem{lem}[thm]{Lemma}

\makeatletter

\def\<{\langle}
\def\>{\rangle}
\def\({\left(}
\def\){\right)}
\def\[{\left[}
\def\]{\right]}

\makeatother
\title{Normal Scalar Curvature Inequality on a Class of Austere Submanifolds}
\author[J.Q. Ge]{Jianquan Ge}
\address{School of Mathematical Sciences, Laboratory of Mathematics and Complex Systems, Beijing Normal
University, Beijing 100875, P.R. CHINA.}
\email{jqge@bnu.edu.cn}

\author[Y. Tao]{Ya Tao}
\address{Chern Institute of Mathematics, Nankai University, Tianjin 300071, P. R. China.}
\email{tao-ya@mail.nankai.edu.cn}

\author[Y. Zhou]{Yi Zhou$^{*}$}
\address{Beijing International Center for Mathematical Research, Peking University,
Beijing 100871, P.R. CHINA.}
\email{yizhou@bicmr.pku.edu.cn}

\date{}
\thanks {$^{*}$ the corresponding author.}
\thanks{J. Q. Ge is partially supported by NSFC (No. 12171037, 12571049) and the Fundamental Research Funds for the Central Universities.}
\thanks{Y. Zhou is partially supported by  NSFC (No. 12171037, 12271040), China Postdoctoral Science Foundation (No. BX20230018)
and National Key R$\&$D Program of China 2020YFA0712800.}

\begin{document}
\maketitle

\begin{abstract}
In this paper, we establish new normal scalar curvature inequalities on a class of austere submanifolds
by proving sharper DDVV-type inequalities on associated austere subspaces.
We also provide some examples of austere submanifolds in this class and point out
one of them achieves the equality in our normal scalar curvature inequality everywhere.
As a byproduct, we obtain a Simons-type gap theorem for closed austere submanifolds in unit spheres
which belong to that class.
\end{abstract}

\section{Introduction}\label{introduction}
The concept of austere submanifolds was introduced by Harvey-Lawson \cite{HL82}
for constructing special Lagrangian submanifolds in $\mathbb{C}^n$.
A submanifold of a Riemannian manifold is said to be austere if
the nonzero principal curvatures in any normal direction occur in oppositely signed pairs.
Clearly, austerity implies minimality, and these two concepts are equivalent only for curves and surfaces.

In \cite{B91}, Bryant initiated the study of austere submanifolds
which began with describing the possible second fundamental forms of austere submanifolds.
Let $M^n$ be a submanifold of a Riemannian manifold $N^{n+m}$,
and let $\operatorname{II}$ be the second fundamental form.
For any $p\in M$, let $|\operatorname{II}_p|\subset S^2(T_p^*M)$ be the subspace spanned by
$\{\operatorname{II}_{\xi}: \xi\in T_p^\bot M\}$,
where $\operatorname{II}_{\xi}(X, Y):=\<\operatorname{II}(X, Y),\xi\>$ denotes the second fundamental form
along the normal direction $\xi$.
By fixing an orthonormal basis, we can identify the vector space $|\operatorname{II}_p|$
with a subspace of the vector space $\mathcal{S}_n$ of $n\times n$ real symmetric matrices.
If a subspace $\mathcal{Q}\subset\mathcal{S}_n$ is consisting of matrices
whose nonzero eigenvalues occur in oppositely signed pairs,
we say that $\mathcal{Q}$ is an austere subspace.
Thus the problem is translated to classify austere subspaces
up to the conjugate action of $\operatorname{O}(n)$ on $\mathcal{S}_n$.
There is an insightful observation that for any orthogonal matrix $P\in\operatorname{O}(n)$,
the subspace $\mathcal{Q}_P:=\{A\in\mathcal{S}_n: PA+AP=0\}$ is austere.
For example, an important class of austere subspaces was obtained by assuming that
$P^2=-I_n$ and $n$ is even.
In fact, almost complex submanifolds in K\"{a}hler manifolds give this class of austere subspaces.
By the conjugate action of $\operatorname{O}(n)$, we can take
$P=\begin{pmatrix}
0 & I_{\frac{n}2}  \\
-I_{\frac{n}2}  & 0 \\
\end{pmatrix}$,
and thus the austere subspace $\mathcal{Q}_P$ is
$$\mathcal{Q}(n):=\left\{\begin{pmatrix}
a_1 & a_2  \\
a_2 & -a_1 \\
\end{pmatrix}: a_1, a_2\in\mathcal{S}_{\frac{n}2}\right\}.$$
Another typical class of austere subspaces came from assuming that $P^2=I_n$ and $P\neq\pm I_n$.
In this case, there exist positive integers $p,q$ such that $P$ conjugates to
$\begin{pmatrix}
I_p & 0  \\
0  & -I_q \\
\end{pmatrix}$,
and the austere subspace $\mathcal{Q}_P$ conjugates to
$$\mathcal{Q}(p,q):=\left\{\begin{pmatrix}
0 & a  \\
a^t & 0 \\
\end{pmatrix}: a\in M(p, q)\right\},$$
where $M(p, q)$ denotes the vector space of $p\times q$ real matrices.
After complicated construction and proof, Bryant classified austere subspaces in $\mathcal{S}_n$ for $n\leq4$. Based on this algebraic classification, Bryant locally classified $3$-dimensional austere submanifolds in Euclidean spaces, Ionel-Ivey \cite{II10, II12} studied $4$-dimensional austere submanifolds in Euclidean spaces and obtained some classification results under additional assumptions.
In addition, $3$-dimensional austere submanifolds in unit spheres and hyperbolic spaces were locally classified by Dajczer-Florit \cite{DF01} and Choi-Lu \cite{CL08}, respectively.
More research on the construction and classification of austere submanifolds can be found in \cite{DV23, GZ23, IST09, KM22}.

Austere submanifolds also come from minimal Wintgen ideal submanifolds which are minimal submanifolds
attaining the following normal scalar curvature inequality everywhere \cite{DT09, XLMW14, XLMW18}.
For an immersed submanifold $M^n$ of a real space form $N^{n+m}(\kappa)$ with constant sectional curvature $\kappa$, De Smet-Dillen-Verstraelen-Vrancken \cite{DDVV99} proposed the so-called DDVV conjecture:
there exists a pointwise normal scalar curvature inequality (also known as DDVV inequality)
\begin{equation}\label{GeoDDVV)}
\rho+\rho^{\bot}\leq \|H\|^2+\kappa,
\end{equation}
where $\rho$ denotes the normalized scalar curvature, $\rho^{\bot}$ denotes the normalized normal scalar curvature and $\|H\|^2$ denotes the mean curvature.
Dillen-Fastenakels-Veken \cite{DFV07} transformed (\ref{GeoDDVV)}) into the so-called DDVV inequality
\begin{equation}\label{AlgDDVV)}
\sum^m_{r,s=1}\|\[B_r,B_s\]\|^2\leq\(\sum^m_{r=1}\|B_r\|^2\)^2
\end{equation}
for $n\times n$ real symmetric matrices $B_1, \ldots, B_m$,
where $[A,B]:=AB-BA$ denotes the commutator and
$\|B\|^2:=\operatorname{tr}(BB^t)$ denotes the squared Frobenius norm.
After some partial results \cite{CL08, DFV07, DHTV04, Lu07},
the DDVV inequality (\ref{AlgDDVV)}) was finally proved by Lu \cite{Lu11} and Ge-Tang \cite{GT08} independently.
In addition, the equality condition in (\ref{AlgDDVV)}) were completely determined in \cite{GT08}.
Hence we know that the equality in (\ref{GeoDDVV)}) holds at some point $p\in M$ if and only if there exist an orthonormal basis $\{e_1,\cdots,e_n\}$ of $T_pM$ and an orthonormal basis $\{\xi_1,\cdots,\xi_m\}$ of $T_p^{\bot}M$ such that
the matrices corresponding to the second fundamental form along directions $\xi_1,\cdots,\xi_m$ are
$$A_1=\lambda_1I_n+\mu\operatorname{diag}(H_1,0),\
A_2=\lambda_2I_n+\mu\operatorname{diag}(H_2,0),\
A_3=\lambda_3I_n$$ and $A_{\xi_r}=0$ for $r>3$,
where $\mu, \lambda_1, \lambda_2, \lambda_3$ are real numbers and $$H_1=
\begin{pmatrix}
1 & 0 \\
0 & -1 \\
\end{pmatrix},\
H_2=
\begin{pmatrix}
0 & 1 \\
1 & 0 \\
\end{pmatrix}.$$
See \cite{GLTZ24, GT11} for detailed surveys on the history and developments of the DDVV inequality.
Among many generalizations of the normal scalar curvature inequality,
Ge-Tang-Yan \cite{GTY20} proved sharper normal scalar curvature inequalities
on the focal submanifolds of isoparametric hypersurfaces in unit spheres.
Note that the isoparametric focal submanifolds are a rich source of austere submanifolds.
It is natural to consider new normal scalar curvature inequalities on more general austere submanifolds.

For a given austere subspace $\mathcal{Q}$,
we say that an austere submanifold $M$ is of type $\mathcal{Q}$ if for each point $p\in M$,
there exists an orthonormal basis of $T_pM$ such that $|\operatorname{II}_p|\subset\mathcal{Q}$.
In some sense, this definition generalizes the three types of $4$-dimensional austere submanifolds
introduced by Ionel-Ivey \cite{II10, II12}.
Since minimal Wintgen ideal submanifolds are austere submanifolds of type $\mathcal{Q}(n)$,
we do not expect new normal scalar curvature inequalities for this class of austere submanifolds.
However, for austere submanifolds of type $\mathcal{Q}(p, q)$,
we can establish sharper normal scalar curvature inequalities by proving the following DDVV-type inequalities.

To describe the equality condition in Theorem \ref{DDVVQ(p,q)} and \ref{DDVVQ(n-1,1)}, we define
a $K(n, m):=\operatorname{O}(n)\times\operatorname{O}(m)$ action on a family of $n\times n$ real matrices $(B_1,\cdots,B_m)$ by
$$(P,R)\cdot(B_1,\cdots,B_m):=\bigg(\sum_{j=1}^mR_{j1}P^tB_jP,\cdots,\sum_{j=1}^mR_{jm}P^tB_jP\bigg),$$ where $P\in\operatorname{O}(n)$ and $R=(R_{jk})\in\operatorname{O}(m)$ acts as a rotation on the matrix tuple $(P^tB_1P,\cdots,P^tB_mP)$.
In addition, we use the notation $E_{ij}$ to denote the $n\times n$ matrix with $1$ in position $(i,j)$ and $0$ elsewhere.

\begin{thm}\label{DDVVQ(p,q)}
Suppose $B_1,\cdots,B_m\in\mathcal{Q}(p,q)$. Then
\begin{equation}\label{ineqQ(p,q)}
\sum^m_{r,s=1}\|\[B_r,B_s\]\|^2\leq \frac12\(\sum^m_{r=1}\|B_r\|^2\)^2.
\end{equation}
The equality holds if and only if there exists a $(P,R)\in K(n, m)$ such that
$$(P,R)\cdot(B_1,\cdots,B_m)=
(\lambda\operatorname{diag}(C_1,0), \lambda\operatorname{diag}(C_2,0), 0, \cdots, 0)$$
for some $\lambda\geq0$, where
$$C_1:=
\begin{pmatrix}
0 & 0 & 1 & 0 \\
0 & 0 & 0 & 1 \\
1 & 0 & 0 & 0\\
0 & 1 & 0 & 0\\
\end{pmatrix},\
C_2:=
\begin{pmatrix}
0 & 0 & 0 & -1\\
0 & 0 & 1  & 0\\
0 & 1 & 0 & 0\\
-1 & 0 & 0 & 0\\
\end{pmatrix}.$$
\end{thm}
The proof follows Ge-Tang's method \cite{GT08} which enables us to derive the equality condition.
Except for the trivial case $B_1=\cdots=B_m=0$,
the equality in the above theorem can only be achieved when $p, q\geq 2$.
For the remaining cases, we have the following sharper inequalities.

\begin{thm}\label{DDVVQ(n-1,1)}
Suppose $B_1,\cdots,B_m\in\mathcal{Q}(n-1,1)$ with $n\geq3$.
Then
\begin{equation}\label{ineqQ(n-1,1)}
\sum^m_{r,s=1}\|\[B_r,B_s\]\|^2\leq \frac{d-1}{2d}\(\sum^m_{r=1}\|B_r\|^2\)^2,
\end{equation}
where $d=\operatorname{dim}\operatorname{Span}\{B_1,\cdots,B_m\}$.
The equality holds if and only if there exists a $(P,R)\in K(n, m)$ such that
$$(P,R)\cdot(B_1,\cdots,B_m)=(\lambda D_1, \lambda D_2, \cdots, \lambda D_d, 0, \cdots, 0)$$
for some $\lambda\geq0$, where $$D_r:=E_{rn}+E_{nr}\ \mbox{for}\ 1\leq r\leq d.$$
\end{thm}

As for the proof of classic DDVV inequality,
Theorem \ref{DDVVQ(p,q)} and \ref{DDVVQ(n-1,1)} imply the following normal scalar curvature inequality.

\begin{thm}\label{GeoDDVVQ(p,q)}
Let $M^n$ be an austere submanifold of type $\mathcal{Q}(p,q)$ in a real space form $N^{n+m}(\kappa)$. Then
\begin{equation}\label{GeoineqQ(p,q)}
\rho^{\bot}\leq \frac{\sqrt{2}}{2}(\kappa-\rho).
\end{equation}
The equality holds at some point $p\in M$ if and only if there exist an orthonormal basis $\{e_1,\cdots,e_n\}$ of $T_pM$ and an orthonormal basis $\{\xi_1,\cdots,\xi_m\}$ of $T_p^{\bot}M$ such that
the matrices corresponding to the second fundamental form along directions $\xi_1,\cdots,\xi_m$ are
$$A_1=\lambda\operatorname{diag}(C_1,0),\ A_2=\lambda\operatorname{diag}(C_2,0)\
\mbox{and}\ A_r=0\ \mbox{for}\ r>2,$$
where $\lambda$ is a real number.
\end{thm}

\begin{thm}\label{GeoDDVVQ(n-1,1)}
Let $M^n$ be an austere submanifold of type $\mathcal{Q}(n-1,1)$ in a real space form $N^{n+m}(\kappa)$. Then for any point $p\in M$,
\begin{equation}\label{GeoineqQ(n-1,1)}
\rho^{\bot}\leq \sqrt{\frac{d-1}{2d}}(\kappa-\rho),
\end{equation}
where $d$ denotes the dimension of $|\operatorname{II}_p|$.
The equality holds if and only if there exist an orthonormal basis $\{e_1,\cdots,e_n\}$ of $T_pM$ and an orthonormal basis $\{\xi_1,\cdots,\xi_m\}$ of $T_p^{\bot}M$ such that
the matrices corresponding to the second fundamental form along directions $\xi_1,\cdots,\xi_m$ are
$$A_r=\lambda D_r\ \mbox{for}\ 1\leq r\leq d\
\mbox{and}\ A_r=0\ \mbox{for}\ r>d,$$
where $\lambda$ is a real number.
\end{thm}

In order to illustrate that the above theorems make sense,
we give some explicit examples of austere submanifold of type $\mathcal{Q}(p,q)$,
including Bryant's generalized helicoids, two austere embeddings of Grassmann manifolds $\operatorname{G}_2(\mathbb{R}^4)$ and $\operatorname{G}_2(\mathbb{C}^4)$,
and some isoparametric focal submanifolds satisfying Condition A in unit spheres.
In fact, among focal submanifolds of isoparametric hypersurfaces with $4$ distinct principal curvatures,
we find that the focal submanifolds $M_+$ with multiplicities $(1, 2k)$ and
the focal submanifold $M_-$ with multiplicities $(4, 3)$(definite)
are austere submanifolds of type $\mathcal{Q}(p, q)$ for some suitable integers $p, q$.
It is worth pointing out that the focal submanifold $M_+$ with multiplicities $(1, 2)$
achieves the equality in Theorem \ref{GeoDDVVQ(p,q)}.

Closely related to the DDVV inequality (\ref{AlgDDVV)}), B\"ottcher-Wenzel \cite{BW05} raised the so-called BW conjecture: Let $X, Y$ be arbitrary $n\times n$ real matrices, then
\begin{equation}\label{BWineq}
\|\[X,Y\]\|^2\leq 2\|X\|^2\|Y\|^2.
\end{equation}
The BW inequality (\ref{BWineq}) was proved by B\"{o}ttcher-Wenzel \cite{BW08}, Vong-Jin \cite{VJ08}, Audenaert \cite{AKMR09} and Lu \cite{Lu11, Lu12} in various ways.
There are also many generalizations \cite{BW08, GLZ21, GLLZ20} and equality characterizations \cite{CVW10}
of BW-type inequalities.
In fact, the BW-type inequality for symmetric matrices is a core technology in the proof of
the Simons' integral inequality by Chern-do Carmo-Kobayashi \cite{CDK}.
Since we can prove a sharper BW-type inequality (Lemma \ref{BWQ(p,q)}) for matrices in $\mathcal{Q}(p,q)$,
we can also establish a Simons-type gap theorem (Theorem \ref{Simons}) for closed austere submanifolds of type $\mathcal{Q}(p,q)$ in unit spheres.

The rest of this paper is organized as follows.
In Section \ref{sec2}, we give the proofs of Theorem \ref{DDVVQ(p,q)}-\ref{GeoDDVVQ(n-1,1)}.
In Section \ref{sec3}, we introduce some austere submanifolds of type $\mathcal{Q}(p,q)$.
At last, in Section \ref{sec4}, we prove the Simons-type gap theorem (Theorem \ref{Simons}).

\section{DDVV-type Inequality and Normal Scalar Curvature Inequality}\label{sec2}
\subsection{Preparatory lemmas}\label{sec2.1}
In this subsection, we will make some preparations for proving Theorem \ref{DDVVQ(p,q)}.
We recall that $\mathcal{Q}(p,q)$ is a real vector space of dimension $pq=:N$.
Without loss of generality, we may assume that $p\geq q$ in the proof of Theorem \ref{DDVVQ(p,q)}.
In order to construct an orthonormal basis of $\mathcal{Q}(p,q)$, we consider the index set
$$S:=\{(i,j): 1\leq i\leq p,\ p+1\leq j\leq n\}.$$
For every $(i,j)\in S$, let
$$\hat{E}_{ij}:=\frac1{\sqrt{2}}(E_{ij}+E_{ji}),$$
where $E_{ij}$ is the $n\times n$ matrix with $1$ in position $(i,j)$ and $0$ elsewhere.
Then $\{\hat{E}_{ij}\}$ is an orthonormal basis of $\mathcal{Q}(p,q)$.
Next, we put an order on the index set $S$ by
\begin{equation}\label{eqno1}
(i,j)<(k,l)\ \text{if and only if}\ i<k\ \text{or}\ i=k, j<l,
\end{equation}
and use this order to identify an index $(i,j)\in S$ with a single (Greek) index
$\alpha\in\{1,\cdots,N\}$, i.e., $(i,j)\ \hat{=}\ \alpha=(i-1)q+(j-p)$.

For $\alpha\ \hat{=}\ (i,j)\leq \beta\ \hat{=}\ (k,l)$ in $S$, direct calculations imply
\begin{equation}\label{eqno2}
\|[\hat{E}_\alpha,\hat{E}_\beta]\|^2=
\begin{cases}
\frac12   &\text{if}\ i=k<j<l\ \text{or}\ i<k<j=l;\\
0         &\text{otherwise},
\end{cases}
\end{equation}
and
\begin{equation}\label{eqno3}
\sum_{\gamma\in S}\<[\hat{E}_\alpha, \hat{E}_\gamma],[\hat{E}_\beta,\hat{E}_\gamma]\>
=\frac12(n-2)\delta_{ik}\delta_{jl},
\end{equation}
where $\<A,B\>:=\operatorname{tr}\(AB^t\)$ denotes the standard inner product.

Suppose $\{\hat{Q}_\alpha\}_{\alpha\in S}$ be any orthonormal basis of $\mathcal{Q}(p,q)$. There exists a unique orthogonal matrix $Q\in\operatorname{O}(N)$ such that
\begin{equation}\label{Qeq}
(\hat{Q}_1, \cdots, \hat{Q}_N)=(\hat{E}_1, \cdots, \hat{E}_N)Q,
\end{equation}
i.e., $\hat{Q}_\alpha=\sum_\beta q_{\beta\alpha}\hat{E}_\beta$ for $Q=(q_{\alpha\beta})_{N\times N}$, and if $\hat{Q}_\alpha=(\hat{q}_{ij}^\alpha)_{n\times n}$, then
\begin{equation}\label{qalpha}
\hat{q}_{ij}^\alpha=\hat{q}_{ji}^\alpha=
\begin{cases}
\frac1{\sqrt{2}}q_{\gamma\alpha}   &\text{if}\ (i,j)\in S;\\
0         &\text{otherwise},
\end{cases}
\end{equation}
where $\gamma\hat{=}(i,j)$.
It follows from (\ref{eqno3}, \ref{Qeq}) that
\begin{equation}\label{lem4}
\sum_{\beta\in S}\|[\hat{Q}_\alpha,\hat{Q}_\beta]\|^2=\frac12(n-2)
\end{equation}
for any $Q\in\operatorname{O}(N)$ and any $\alpha\in S$.

In the last part of this subsection, we prove some technical lemmas.
Let $\lambda_1, \cdots, \lambda_q$ be real numbers satisfying $\sum_{i=1}^q\lambda_i^2=\frac12$ and $\lambda_1\geq\cdots\geq\lambda_q\geq0$.
Let $$I:=\left\{(i,j): 1\leq i< j\leq q,\ \lambda_i+\lambda_j>\frac1{\sqrt{2}}\right\},$$
and let $n_0$ be the number of elements in $I$.

\begin{lem}\label{lem1}
If $n_0\geq1$, then $I=\{(1,j): 2\leq j\leq n_0+1\}$ or $\{(1,2), (1,3), (2,3)\}$.
\end{lem}
\begin{proof}
If $n_0=1$, then the single element must be $(1,2)$.
If $n_0\geq2$, let $(i_1,j_1)$ and $(i_2,j_2)$ be two different elements of $I$.
We assert that $i_1$, $j_1$, $i_2$ and $j_2$ would not be four distinct elements.
Otherwise,
$$\frac12\geq\lambda^2_{i_1}+\lambda^2_{j_1}+\lambda^2_{i_2}+\lambda^2_{j_2}\geq
\frac{1}2(\lambda_{i_1}+\lambda_{j_1})^2+\frac{1}2(\lambda_{i_2}+\lambda_{j_2})^2>\frac12$$
is a contradiction.
Since we must have $(1,2), (1,3)\in I$, the above assertion implies either $I=\{(1,j): 2\leq j\leq n_0+1\}$ or $I=\{(1,2), (1,3), (2,3)\}$.
\end{proof}

\begin{lem}\label{lem2}
We have $\sum_{(i,j)\in I}\[(\lambda_i+\lambda_j)^2-\frac12\]\leq\frac12,$ where the equality holds
if and only if $n_0=1$, $\lambda_1=\lambda_2=\frac12, \lambda_3=\cdots=\lambda_q=0$ or
$n_0=3$, $\lambda_1=\lambda_2=\lambda_3=\frac1{\sqrt{6}}, \lambda_4=\cdots=\lambda_q=0$.
\end{lem}
\begin{proof}
By Lemma \ref{lem1}, we divide the proof into two cases:

$(1)$ If $I=\{(1,j): 2\leq j\leq n_0+1\}$, then
$$\begin{aligned}
\sum_{(i,j)\in I}\left[(\lambda_i+\lambda_j)^2-\frac12\right]
&=\sum_{j=2}^{n_0+1}(\lambda^2_1+\lambda^2_j+2\lambda_1\lambda_j)-\frac12n_0\\
&=n_0\lambda^2_1+\sum_{j=2}^{n_0+1}\lambda^2_j+2\lambda_1\sum_{j=2}^{n_0+1}\lambda_j-\frac12n_0\\
&\leq(n_0+1)\lambda^2_1+\sum_{j=2}^{n_0+1}\lambda^2_j+\bigg(\sum_{j=2}^{n_0+1}\lambda_j\bigg)^2-\frac12n_0\\
&\leq(n_0+1)\Big(\lambda^2_1+\sum_{j=2}^{n_0+1}\lambda^2_j\Big)-\frac12n_0\\
&\leq(n_0+1)\sum^{q}_{i=1}\lambda^2_i-\frac12n_0\\
&=\frac12(n_0+1)-\frac12n_0=\frac12.
\end{aligned}$$

$(2)$ If $I=\{(1,2), (1,3), (2,3)\}$, then
$$\begin{aligned}
\sum_{(i,j)\in I}\left[(\lambda_i+\lambda_j)^2-\frac12\right]
&=\sum_{(i,j)\in I}(\lambda_i+\lambda_j)^2-\frac32
\leq2\sum_{(i,j)\in I}(\lambda^2_i+\lambda^2_j)-\frac32\\
&=4\sum_{i=1}^3\lambda^2_i-\frac32\leq2-\frac32=\frac12.
\end{aligned}$$
The equality condition is easily seen from the proof.
\end{proof}

\begin{lem}\label{lem3}
For any $Q\in\operatorname{O}(N)$, given any $\alpha\in S$ and any subset $J_\alpha\subset S$, we have $$\sum_{\beta\in{J}_\alpha}\bigg(\|[\hat{Q}_\alpha,\hat{Q}_\beta]\|^2-\frac12\bigg)\leq\frac12.$$
\end{lem}
\begin{proof}
Since the above inequality is invariant under orthogonal congruences,
we can transform $\hat{Q}_\alpha$ into an appropriate form.
For the matrix $\hat{Q}_\alpha=\begin{pmatrix}
0 & a  \\
a^t & 0 \\
\end{pmatrix}$,
there exist orthogonal matrices $U\in\operatorname{O}(p)$ and $V\in\operatorname{O}(q)$ such that
$UaV=\begin{pmatrix}
D  \\
0 \\
\end{pmatrix}$, where $D=\operatorname{diag}(\lambda_1, \cdots, \lambda_q)$ with $\sum_{i=1}^q\lambda_i^2=\frac12$ and $\lambda_1\geq\cdots\geq\lambda_q\geq0$.
If we take $P_1=\operatorname{diag}(U^t, V)\in\operatorname{O}(n)$,
then $$P_1^t\hat{Q}_\alpha P_1=\begin{pmatrix}
0 & 0 & D \\
0 & 0 & 0 \\
D & 0 & 0 \\
\end{pmatrix},$$
and $(P_1^t\hat{Q}_1P_1, \cdots, P_1^t\hat{Q}_NP_1)$ is still an orthonormal basis of $\mathcal{Q}(p,q)$.
Hence we can assume without loss of generality that
$$\hat{Q}_\alpha=\begin{pmatrix}
0 & 0 & D \\
0 & 0 & 0 \\
D & 0 & 0 \\
\end{pmatrix},\
\hat{Q}_\beta=\begin{pmatrix}
0 & 0 & b_\beta \\
0 & 0 & c_\beta \\
b_\beta^t & c_\beta^t & 0 \\
\end{pmatrix}\ \mbox{for}\ \beta\neq\alpha,$$
where $b_\beta=\big(b_{ij}^\beta\big)_{q\times q}$, $c_\beta=\big(c_{ij}^\beta\big)_{(p-q)\times q}$.
Next, take $$P_2=\begin{pmatrix}
\frac1{\sqrt{2}}I_q & 0 & \frac1{\sqrt{2}}I_q \\
0 & I_{p-q} & 0 \\
\frac1{\sqrt{2}}I_q & 0 & -\frac1{\sqrt{2}}I_q \\
\end{pmatrix}\in\operatorname{O}(n).$$
Then we have $$P_2^t\hat{Q}_\alpha P_2=\begin{pmatrix}
D & 0 & 0 \\
0 & 0 & 0 \\
0 & 0 & -D \\
\end{pmatrix},\
P_2^t\hat{Q}_\beta P_2=\begin{pmatrix}
X_\beta & \frac1{\sqrt{2}}c^t_\beta & Y^t_\beta \\
\frac1{\sqrt{2}}c_\beta & 0 & -\frac1{\sqrt{2}}c_\beta \\
Y_\beta & -\frac1{\sqrt{2}}c^t_\beta & -X_\beta \\
\end{pmatrix},$$
where $X_\beta=\frac12(b_\beta+b_\beta^t)$ and $Y_\beta=\frac12(b_\beta-b_\beta^t)$
denote the symmetric and skew-symmetric parts of $b_\beta$, respectively.
Direct calculations show that
$$[P_2^t\hat{Q}_\alpha P_2,P_2^t\hat{Q}_\beta P_2]=\begin{pmatrix}
\[D, X_\beta\] & \frac1{\sqrt{2}}Dc^t_\beta & DY^t_\beta+Y^t_\beta D \\
-\frac1{\sqrt{2}}c_\beta D & 0 & -\frac1{\sqrt{2}}c_\beta D  \\
DY^t_\beta+Y^t_\beta D & \frac1{\sqrt{2}}Dc^t_\beta & \[D, X_\beta\] \\
\end{pmatrix}$$
and thus
$$\begin{aligned}
\|[\hat{Q}_\alpha,\hat{Q}_\beta]\|^2
&=\|[P_2^t\hat{Q}_\alpha P_2,P_2^t\hat{Q}_\beta P_2]\|^2=
2\left\|\[D, X_\beta\]\right\|^2+2\left\|DY_\beta+Y_\beta D\right\|^2+2\left\|c_\beta D\right\|^2\\
&=2\sum_{i,j=1}^q(\lambda_i-\lambda_j)^2\big(x_{ij}^\beta\big)^2
+2\sum_{i,j=1}^q(\lambda_i+\lambda_j)^2\big(y_{ij}^\beta\big)^2
+2\sum_{i=1}^{p-q}\sum_{j=1}^q\lambda_j^2\big(c_{ij}^\beta\big)^2
\end{aligned}$$
for $X_\beta=(x_{ij}^\beta\big)_{q\times q}$, $Y_\beta=(y_{ij}^\beta\big)_{q\times q}$.
It is easy to see that
$$\begin{aligned}
\frac12&=\left\|b_\beta\right\|^2+\left\|c_\beta\right\|^2
=\left\|X_\beta\right\|^2+\left\|Y_\beta\right\|^2+\left\|c_\beta\right\|^2\\
&=\sum_{i,j=1}^q\big(x_{ij}^\beta\big)^2+\sum_{i,j=1}^q\big(y_{ij}^\beta\big)^2
+\sum_{i=1}^{p-q}\sum_{j=1}^q\big(c_{ij}^\beta\big)^2,
\end{aligned}$$
and hence
$$\begin{aligned}
\|[\hat{Q}_\alpha,\hat{Q}_\beta]\|^2-\frac12
&=2\sum_{i,j=1}^q\[(\lambda_i-\lambda_j)^2-\frac12\]\big(x_{ij}^\beta\big)^2
+2\sum_{i,j=1}^q\[(\lambda_i+\lambda_j)^2-\frac12\]\big(y_{ij}^\beta\big)^2\\
&\quad+2\sum_{i=1}^{p-q}\sum_{j=1}^q\Big(\lambda_j^2-\frac12\Big)\big(c_{ij}^\beta\big)^2.
\end{aligned}$$
Since $\sum_{i=1}^q\lambda_i^2=\frac12$ and $\lambda_1\geq\cdots\geq\lambda_q\geq0$
imply that $(\lambda_i-\lambda_j)^2-\frac12\leq0$ and $\big(\lambda_j^2-\frac12\big)\leq0$
for any $1\leq i, j\leq q$, we have
\begin{equation}\label{ineq-lem3-1}
\|[\hat{Q}_\alpha,\hat{Q}_\beta]\|^2-\frac12
\leq4\sum_{(i,j)\in I}\[(\lambda_i+\lambda_j)^2-\frac12\]\big(y_{ij}^\beta\big)^2.
\end{equation}
By (\ref{qalpha}), $y_{ij}^\beta=\frac12\big(b_{ij}^\beta-b_{ji}^\beta\big)
=\frac12\big(\hat{q}_{i(j+p)}^\beta-\hat{q}_{j(i+p)}^\beta\big)
=\frac1{2\sqrt{2}}\(q_{\gamma\beta}-q_{\tau\beta}\),$
where $\gamma\hat{=}(i,j+p)$, $\tau\hat{=}(j,i+p)$.
Then $Q\in\operatorname{O}(N)$ implies that for any $1\leq i< j\leq q$,
\begin{equation}\label{ineq-lem3-2}
\sum_{\beta\in J_\alpha}\big(y_{ij}^\beta\big)^2
\leq\sum_{\beta\in S}\big(y_{ij}^\beta\big)^2
=\frac18\sum_{\beta\in S}\(q_{\gamma\beta}^2+q_{\tau\beta}^2-2q_{\gamma\beta}q_{\tau\beta}\)=\frac14.
\end{equation}
Combining with (\ref{ineq-lem3-1}) and Lemma \ref{lem2}, we have
\begin{equation}\label{ineq-lem3}
\begin{aligned}
\sum_{\beta\in J_\alpha}\(\|[\hat{Q}_\alpha,\hat{Q}_\beta]\|^2-\frac12\)
&\leq4\sum_{\beta\in J_\alpha}\sum_{(i,j)\in I}\[(\lambda_i+\lambda_j)^2-\frac12\]
\big(y_{ij}^\beta\big)^2\\
&=4\sum_{(i,j)\in I}\[(\lambda_i+\lambda_j)^2-\frac12\]
\sum_{\beta\in J_\alpha}\big(y_{ij}^\beta\big)^2\\
&\leq\sum_{(i,j)\in I}\[(\lambda_i+\lambda_j)^2-\frac12\]\leq\frac12.
\end{aligned}
\end{equation}
\end{proof}


\subsection{Proof of Theorem \ref{DDVVQ(p,q)}}\label{sec2.3}
In this subsection, we are going to prove Theorem \ref{DDVVQ(p,q)} by Ge-Tang's method \cite{GT08}.
The first step is to transform the DDVV-type inequality to the inequality (\ref{DDVVtransf})
of quadratic forms.
To do this, we need to use the following ``multiplicative" map as in \cite{Ge14, GT08, GXYZ17}.
Let $\varphi:M(m,n)\longrightarrow M(\binom{m}{2},\binom{n}{2})$ be the map defined by $\varphi\(A\)_{\(i,j\)\(k,l\)}:=A\binom{kl}{ij}$, where $1\leq i<j\leq m$, $1\leq k<l\leq n$, and $A\binom{kl}{ij}=a_{ik}a_{jl}-a_{il}a_{jk}$ is the discriminant of the $2\times2$ submatrix of $A$,
that is the intersection of rows $i$ and $j$ with columns $k$ and $l$,
arranged with the same order as in (\ref{eqno1}). One can verify directly that $\varphi\(I_n\)=I_{\binom{n}{2}}$, $\varphi\(A\)^t=\varphi\(A^t\)$, and
$\varphi(AB)=\varphi(A)\varphi(B)$ holds for any $A\in M(m,k)$ and any $B\in M(k,n)$.

Let $B_1,\cdots,B_m\in\mathcal{Q}(p,q)$ with $p+q=n\geq3$.
(If $n=2$, then $\operatorname{dim}\mathcal{Q}(p,q)=1$ and thus inequality (\ref{ineqQ(p,q)}) is trivial.)
Their coefficients in the standard basis $\{\hat{E}_\alpha: \alpha\in S\}$ of $\mathcal{Q}(p,q)$ are determined by a matrix $B\in M\(N,m\)$ as
\begin{equation}\label{BiB}
(B_1,\cdots,B_m)=(\hat{E}_1,\cdots,\hat{E}_N)B.
\end{equation}
Since $B$ is real and $BB^t$ is a $N\times N$ positive semi-definite matrix, there
exists an orthogonal matrix $Q\in\operatorname{SO}(N)$ such that
\begin{equation}\label{BBdiag}
BB^t=Q\operatorname{diag}(x_1,\cdots,x_N)Q^t,
\end{equation}
where $x_{\alpha}\geq0$ for each $1\leq\alpha\leq N$.
Thus
\begin{equation}\label{BBnorm}
\sum_{r=1}^m\|B_r\|^2=\|B\|^2 =\sum_{\alpha=1}^Nx_{\alpha}.
\end{equation}
Moreover, this orthogonal matrix $Q$ determines an orthonormal basis $\{\hat{Q}_{\alpha}: \alpha\in S\}$ of $\mathcal{Q}(p,q)$ as (\ref{Qeq}) in Section \ref{sec2}.
We use the lexicographic order as in (\ref{eqno1}) for the indices sets $\{\(r,s\): 1\leq r<s\leq m\}$ and $\{\(\alpha,\beta\): 1\leq\alpha<\beta\leq N\}$. Then we can arrange $\{[B_r,B_s]\}_{r<s}$ and $\{[\hat{E}_\alpha,\hat{E}_\beta]\}_{\alpha<\beta}$ into $\binom{m}{2}$- and $\binom{N}{2}$-vectors, respectively.
Now we observe that
\begin{equation}\label{comm-bases}
\Big([B_1, B_2], \cdots, [B_{m-1}, B_m]\Big)=
\Big([\hat{E}_1, \hat{E}_2], \cdots, [\hat{E}_{N-1}, \hat{E}_N]\Big)\varphi(B).
\end{equation}
Let $C(E)$ denote the $\binom{N}{2}\times\binom{N}{2}$ matrix defined by
\begin{equation}\label{comm-bases1}
C(E)_{(\alpha,\beta)(\gamma,\tau)}:=
\langle[\hat{E}_{\alpha}, \hat{E}_{\beta}], [\hat{E}_{\gamma}, \hat{E}_{\tau}]\rangle,
\end{equation}
for $1\leq\alpha<\beta\leq N$,
$1\leq\gamma<\tau\leq N$. Moreover we will use the same notation for
$\{B_r\}$ and $\{\hat{Q}_{\alpha}\}$, \emph{i.e.}, the $\binom{m}{2}\times\binom{m}{2}$ matrix $C(B):=\Big(\langle
[B_{r_1}, B_{s_1}], [B_{r_2}, B_{s_2}]\rangle\Big)$ and the $\binom{N}{2}\times\binom{N}{2}$ matrix $C(Q):=\Big(\langle[\hat{Q}_{\alpha}, \hat{Q}_{\beta}], [\hat{Q}_{\gamma}, \hat{Q}_{\tau}]\rangle\Big)$
respectively. Then it is obvious from (\ref{comm-bases}) that
\begin{equation}\label{comm-bases2}
C(B)=\varphi(B^t)C(E)\varphi(B), \hskip 0.3cm C(Q)=\varphi(Q^t)C(E)\varphi(Q).
\end{equation}
By (\ref{comm-bases2}) and the multiplicativity of $\varphi$, we have
\begin{eqnarray}\label{commtransf}
\sum_{r,s=1}^m\|[B_r, B_s]\|^2
&=&2\operatorname{tr}C(B)=2\operatorname{tr}\varphi(B^t)C(E)\varphi(B)\\
&=&2\operatorname{tr}\varphi(BB^t)C(E)=2\operatorname{tr}
\varphi(\operatorname{diag}(x_1,\cdots,x_N))C(Q)\nonumber\\
&=&\sum_{\alpha,\beta=1}^Nx_{\alpha}x_{\beta}\|[\hat{Q}_{\alpha},\hat{Q}_{\beta}]\|^2.\nonumber
\end{eqnarray}
Combining (\ref{BBnorm}, \ref{commtransf}), the inequality (\ref{ineqQ(p,q)}) of Theorem \ref{DDVVQ(p,q)} is now transformed into the following:
\begin{equation}\label{DDVVtransf}
\sum_{\alpha,\beta=1}^Nx_\alpha x_\beta\|[\hat{Q}_\alpha,\hat{Q}_\beta]\|^2
\leq\frac12\(\sum_{\alpha=1}^Nx_\alpha\)^2\
\mbox{for any}\ x\in\mathbb{R}^N_+\ \mbox{and any} \ Q\in\operatorname{SO}(N),
\end{equation}
where $\mathbb{R}^N_+:=\left\{x=\(x_1,\cdots,x_N\)\in\mathbb{R}^N\setminus\{0\}:
x_\alpha\geq0\ \mbox{for any}\ 1\leq\alpha\leq N\right\}$.

The next step is to show (\ref{DDVVtransf}). For this, define the function
\begin{equation}\label{quadratic-form}
f_Q(x)=F(x,Q):=\sum_{\alpha,\beta=1}^Nx_{\alpha}x_{\beta}\|[\hat{Q}_{\alpha},
\hat{Q}_{\beta}]\|^2-\frac12\(\sum_{\alpha=1}^Nx_{\alpha}\)^2.
\end{equation}
Then $F$ is a continuous function defined on $\mathbb{R}^N\times\operatorname{SO}(N)$
(equipped with the product topology of Euclidean and Frobenius norm spaces) and thus uniformly continuous on any compact subset of $\mathbb{R}^N\times\operatorname{SO}(N)$.
Let $\Delta:=\big\{x\in\mathbb{R}^N_{+}: \sum_{\alpha=1}^Nx_{\alpha}=1\big\}$
and for any sufficiently small $\varepsilon>0$,
$\Delta_{\varepsilon}:=\left\{x\in\Delta: x_{\alpha}\geq\varepsilon, 1\leq\alpha\leq N\right\}$.
Also let
$$G:=\left\{Q\in\operatorname{SO}(N): f_Q(x)\leq0\ \mbox{for any}\ x\in\Delta\right\},$$
$$G_{\varepsilon}:=\left\{Q\in\operatorname{SO}(N):
f_Q(x)<0\ \mbox{for any}\ x\in\Delta_{\varepsilon}\right\}.$$
We claim that $G=\lim_{\varepsilon\rightarrow0}G_{\varepsilon}=\operatorname{SO}(N).$
Note that this implies (\ref{DDVVtransf}) by the homogeneity of $f_Q$ and thus proves Theorem \ref{DDVVQ(p,q)}.
In fact we can show
\begin{equation}\label{G epsilon}
G_{\varepsilon}=\operatorname{SO}(N)\ \mbox{for any sufficiently small } \varepsilon>0.
\end{equation}
To prove (\ref{G epsilon}), we use the continuity method, in which
we must prove the following three properties and remember that there are only the two trivial sets that are open and closed at the same time:
\begin{itemize}
\item[\textbf{(a)}]\label{step1} $I_N\in G_{\varepsilon}$ (and thus
$G_{\varepsilon}\neq\varnothing$);
\item[\textbf{(b)}]\label{step2} $G_{\varepsilon}$ is open in $\operatorname{SO}(N)$;
\item[\textbf{(c)}]\label{step3} $G_{\varepsilon}$ is closed in $\operatorname{SO}(N)$.
\end{itemize}

\textbf{Proof of (a)}.  For any $ x\in\Delta_\epsilon$, it follows from (\ref{eqno2}) that
$$\begin{aligned}
f_{I_N}\(x\)&=\sum_{\alpha,\beta=1}^Nx_\alpha{x}_\beta\|[\hat{E}_\alpha,\hat{E}_\beta]\|^2-\frac12\\
&=\sum_{\substack{1\leq i\leq p,\\ p+1\leq j<l\leq n}}x_{ij}x_{il}
+\sum_{\substack{1\leq i<k\leq p,\\ p+1\leq j\leq n}}x_{ij}x_{kj}
-\frac12\Bigg(\sum_{i=1}^p\sum_{j=p+1}^nx_{ij}\Bigg)^2\\
&<-\frac12\sum_{i=1}^p\sum_{j=p+1}^nx_{ij}^2<0.
\end{aligned}$$
Hence we have $ I_N\in G_\varepsilon$.  \hfill  $\qed$

\textbf{Proof of (b)}.
Since $F$ is uniformly continuous on
$\triangle_{\varepsilon}\times\operatorname{SO}(N)$, the function $g(Q):=\max_{x\in\Delta_{\varepsilon}}F(x,Q)$ is continuous on $Q\in\operatorname{SO}(N)$ and thus $G_{\varepsilon}$ is obviously open as the preimage of an open set $(-\infty, 0)$ of $g$. \hfill  $\qed$

\textbf{Proof of (c)}.
We only need to prove the following \textbf{a priori estimate}: \textit{Suppose $f_Q\(x\)\leq0$ for every $x\in\Delta_\varepsilon$. Then $f_Q\(x\)<0$ for every $x\in\Delta_\varepsilon$.} Provided with this, for a sequence $\{Q_k\}\subset G_{\varepsilon}$ which converges to a $Q\in\operatorname{SO}(N)$, we have
$$g(Q)=\lim_{k\rightarrow\infty}g(Q_k)=\lim_{k\rightarrow\infty}\max_{x\in\Delta_{\varepsilon}}F(x,Q_k)\leq0.$$
Therefore, $f_Q\(x\)\leq g(Q)\leq0$ for every $x\in\Delta_\varepsilon$. Then $f_Q\(x\)<0$ for every $x\in\Delta_\varepsilon$ and thus $Q\in G_{\varepsilon}$, proving the closedness of $G_{\varepsilon}$.

The proof of this estimate is as follows: If there is a point $y\in\Delta_\varepsilon$ such that $f_Q\(y\)=0$, we can arrange the entries decreasingly and assume without loss of generality that for some $1\leq\gamma\leq N$,
$$y\in\Delta_\varepsilon^\gamma:=\left\{x\in\Delta_\varepsilon: x_\alpha>\varepsilon\ \mbox{for}\ \alpha\leq\gamma,\ \mbox{and}\ x_\beta=\varepsilon\ \mbox{for}\ \beta>\gamma\right\}.$$
With the given prerequisites, then $y$ is a maximum point of $f_Q\(x\)$ in the cone spanned by $\Delta_\varepsilon$ and an interior maximum point of $f_Q\(x\)$ in $\Delta_\varepsilon^\gamma$.
Hence, applying the Lagrange Multiplier Method,
there exist real numbers $a, b_{\gamma+1}, \cdots, b_N$ such that
\begin{equation}\label{partial f}
\begin{aligned}\(\frac{\partial f_Q}{\partial x_1}(y),\cdots,\frac{\partial f_Q}{\partial x_{\gamma}}(y)\)
&=2a(1,\cdots,1); \\
\(\frac{\partial f_Q}{\partial x_{\gamma+1}}(y),\cdots,\frac{\partial f_Q}{\partial x_{N}}(y)\)
&=2(b_{\gamma+1},\cdots,b_N),
\end{aligned}
\end{equation}
or equivalently,
\begin{equation}\label{partial f2}
\sum_{\beta=1}^Ny_{\beta}\|[\hat{Q}_{\alpha},\hat{Q}_{\beta}]\|^2-\frac12=\Big\{
\begin{array}{ll}
a &\text{if}\  \alpha\leq\gamma;\\
b_{\alpha} &\text{if}\  \alpha>\gamma.
\end{array}
\end{equation}
Hence
\begin{equation}\label{partial f3}
f_Q(y)=\Big(\sum_{\alpha=1}^{\gamma}y_{\alpha}\Big)a+\Big(\sum_{\alpha=\gamma+1}^Nb_{\alpha}\Big)\varepsilon
=0\ \mbox{and}\ \sum_{\alpha=1}^{\gamma}y_{\alpha}+(N-\gamma)\varepsilon=1.
\end{equation}
Meanwhile, by the homogeneity of $f_Q$, we see
\begin{equation}\label{partial fv}
\frac{\partial f_Q}{\partial \nu}(y)=2(a\gamma+\sum_{\alpha=\gamma+1}^Nb_{\alpha})\leq 0,
\end{equation} where
$\nu=(1,\cdots,1)$ is the vector normal to $\Delta$ in
$\mathbb{R}^N$. For any sufficiently small $\varepsilon$ (such as
$\varepsilon<1/N$), it follows from (\ref{partial f3}) and (\ref{partial fv}) that
$a\geq 0$. Without loss of generality, we assume
$y_1=\max\{y_1,\cdots,y_{\gamma}\}>\varepsilon$.
Let $$J:=\left\{\beta\in S: \|[\hat{Q}_{1}, \hat{Q}_{\beta}]\|^2\geq\frac12\right\},$$
and let $n_1$ be the number of elements of $J$.
Now combining Lemma \ref{lem3} and (\ref{partial f2}) will give a
contradiction as follows:

\begin{equation}\label{contrad}
\begin{aligned}
\frac12\leq\frac12+a&=\sum^N_{\beta=2}y_\beta\|[\hat{Q}_1,\hat{Q}_\beta]\|^2\\
&=\sum_{\beta\in J}y_\beta\bigg(\|[\hat{Q}_1,\hat{Q}_\beta]\|^2-\frac12\bigg)+\frac12\sum_{\beta\in J}y_\beta+\sum_{\beta\in S\setminus J}y_\beta\|[\hat{Q}_1,\hat{Q}_\beta]\|^2\\
&\leq y_1\sum_{\beta\in J}\bigg(\|[\hat{Q}_1,\hat{Q}_\beta]\|^2-\frac12\bigg)
+\frac12\sum_{\beta\in J}y_\beta+\sum_{\beta\in S\setminus J}y_\beta\|[\hat{Q}_1,\hat{Q}_\beta]\|^2\\
&\leq\frac12y_1+\frac12\sum_{\beta\in J}y_\beta+\sum_{\beta\in S\setminus J}y_\beta\|[\hat{Q}_1,\hat{Q}_\beta]\|^2\\
&\leq\frac12\sum^N_{\beta=1}y_\beta=\frac12.
\end{aligned}
\end{equation}
Thus $a=0$, the third line shows that $y_{\beta}=y_1$ for $\beta\in J$ and the fourth line shows that $\sum_{\beta\in J}\(\|[\hat{Q}_1,\hat{Q}_\beta]\|^2-\frac12\)=\frac12$.
Then by (\ref{lem4}), we have
\begin{equation}\label{n1N}
\frac12(n_1+1)=\sum_{\beta\in J}\|[\hat{Q}_1, \hat{Q}_{\beta}]\|^2
\leq\sum_{\beta\in S}\|[\hat{Q}_1, \hat{Q}_{\beta}]\|^2=\frac12(n-2)<\frac12N.
\end{equation}
Hence $S\setminus(J\cup\{1\})\neq\varnothing$, and the last ``$\leq$" in
(\ref{contrad}) should be ``$<$" by the definition of $J$ and the
positivity of $y_{\beta}$ for $\beta\in S\setminus(J\cup\{1\})$.\hfill
$\Box$

Finally, we consider the equality condition of (\ref{ineqQ(p,q)}) in view of the proof of the a priori estimate. When $f_Q(y)=0$ for some $y\in\Delta$ and $Q\in\operatorname{SO}(N)$,
we can assume without loss of generality that
$$y\in\Delta^\gamma:=\left\{x\in\Delta: x_\alpha>0\ \mbox{for}\ \alpha\leq\gamma,\ \mbox{and}\ x_\beta=0
\ \mbox{for}\ \beta>\gamma\right\}$$
for some $2\leq\gamma\leq N$. Note that the fifth line in (\ref{contrad}) implies that
$y_{\beta}=0$ for all $\beta\in S\setminus(J\cup\{1\})$, and thus $\gamma= n_1+1$.
Therefore, we have also (\ref{partial f}-\ref{n1N}) with $\varepsilon=a=0$, all inequalities in (\ref{contrad}) and thus in Lemma \ref{lem3} achieve the equality.
As in the proof of Lemma \ref{lem3}, we assume without loss of generality that
$$\hat{Q}_1=\begin{pmatrix}
0 & 0 & D \\
0 & 0 & 0 \\
D & 0 & 0 \\
\end{pmatrix},\
\hat{Q}_\beta=\begin{pmatrix}
0 & 0 & b_\beta \\
0 & 0 & c_\beta \\
b_\beta^t & c_\beta^t & 0 \\
\end{pmatrix}\ \mbox{for}\ \beta\in S\setminus\{1\},$$
where $D=\operatorname{diag}(\lambda_1,\ldots, \lambda_q)$.
In the proof of Lemma \ref{lem3} with $J_{\alpha}=J$,
the first line in (\ref{ineq-lem3}) shows that
$c_\beta=0$, $b_\beta$ is skew-symmetric and $b^\beta_{ij}=-b^\beta_{ji}=0$ for all $\beta\in J$
and $(i,j)\in\{(i,j): 1\leq i< j\leq q\}\setminus I$;
the third line in (\ref{ineq-lem3}) shows that
$b^\beta_{ij}=0$ for all $\beta\in S\setminus(J\cup\{1\})$ and $(i,j)\in I$.
For the eigenvalues $\pm\lambda_1,\ldots, \pm\lambda_q$ of $\hat{Q}_1$, because of the equality condition of Lemma \ref{lem2}, we have to consider the following cases:

$(1)$ $\lambda_1=\lambda_2=\frac12$, $\lambda_3=\cdots=\lambda_q=0$, $n_0=1$ and $I=\{(1,2)\}$.
Since $\{\hat{Q}_\alpha\}_{\alpha\in S}$ is an orthonormal basis, by the above discussion,
we have $n_1=1$, $y=(\frac12, \frac12, 0, \cdots, 0)$ 
and $b_2=\pm\frac12(E_{12}-E_{21})$ for the only element $2\in J$.

$(2)$ $\lambda_1=\lambda_2=\lambda_3=\frac1{\sqrt{6}}$, $\lambda_4=\cdots=\lambda_q=0$, $n_0=3$ and $I=\{(1,2), (1,3), (2,3)\}$.
The above discussion imply that for each $\beta\in J$, $b_\beta$ is a linear combination of
$\check{E}_{12}:=\frac12(E_{12}-E_{21})$, $\check{E}_{13}:=\frac12(E_{13}-E_{31})$ and $\check{E}_{23}:=\frac12(E_{23}-E_{32})$.
Direct calculations show that $\|[\hat{Q}_1,\hat{Q}_\beta]\|^2=\frac23$ for each $\beta\in J$, and thus
$\sum_{\beta\in J}\(\|[\hat{Q}_1,\hat{Q}_\beta]\|^2-\frac12\)=n_1\(\frac23-\frac12\)=\frac16n_1$.
Hence we have $n_1=3$, $y=(\frac14, \frac14, \frac14, \frac14, 0, \cdots, 0)$
and $(b_2, b_3, b_4)=(\check{E}_{12}, \check{E}_{13}, \check{E}_{23})P$ for some $P\in\operatorname{O}(3)$.
We can calculate directly that
$\|[\hat{Q}_{\alpha}, \hat{Q}_{\beta}]\|^2=\frac14$ for any $2\leq\alpha\neq\beta\leq4$.
It follows that $$\begin{aligned}
\sum_{\beta=1}^Ny_{\beta}\|[\hat{Q}_2, \hat{Q}_{\beta}]\|^2
&=y_1\|[\hat{Q}_2, \hat{Q}_1]\|^2+\sum_{\beta=3}^4y_{\beta}\|[\hat{Q}_2, \hat{Q}_{\beta}]\|^2\\
&=\frac14\cdot\frac23+2\cdot\frac14\cdot\frac14=\frac7{24}<\frac12,
\end{aligned}$$
which contradicts the equality $(\ref{partial f2})$ for $\alpha=2<4=\gamma$.

So far we have proven that  $f_Q(x)=0$ ($x\in\mathbb{R}^N_+$) if and only if in the orthonormal basis $\{\hat{Q}_{\alpha}\}_{\alpha\in S}$ of $\mathcal{Q}(p,q)$ determined by $Q\in\operatorname{O}(N)$, there exist two of them, say $\hat{Q}_1,\hat{Q}_2$, such that $x_1=x_2\geq0$, $x_{\alpha}=0$ for $3\leq\alpha\leq N$, and for some $P\in\operatorname{O}(n)$, $P^t\hat{Q}_iP=\frac12\operatorname{diag}(C_i,0)$ for $i=1,2$ up to a rotation of the three matrices.

Now let $B_1,\cdots,B_m\in\mathcal{Q}(p,q)$ be matrices achieving the equality with the corresponding matrices $B\in M(N,m)$, $Q\in\operatorname{O}(N)$ satisfying (\ref{Qeq}, \ref{BiB}, \ref{BBdiag})
and the equality conditions in the last paragraph.
Then we have $Q^tBB^tQ=\operatorname{diag}(cI_2,0)$ for some $c>0$
($c=0$ only if $B_i=0$ for all $i=1,\ldots,m$).
Hence $Q^tB=\sqrt{c}(a_1,a_2,0,\cdots,0)^t\in M(N,m)$ for some orthonormal column vectors $a_1,a_2\in\mathbb{R}^m$.
Thus there is an orthogonal matrix $R\in\operatorname{O}(m)$
(by extending $\{a_1,a_2\}$ to an orthonormal basis $\{a_i\}_{i=1}^m$ of $\mathbb{R}^m$
and taking $R=(a_1,\cdots,a_m)\in\operatorname{O}(m)$) such that
$$Q^tBR=\sqrt{c}\begin{pmatrix}
I_2 & 0 \\
0 & 0 \\
\end{pmatrix}\in M(N,m).$$
It follows from (\ref{Qeq}, \ref{BiB}) that
$$\(B_1,\cdots,B_m\)R=\(\hat{E}_1,\cdots,\hat{E}_N\)BR=\(\hat{Q}_1,\cdots,\hat{Q}_N\)Q^tBR
=\sqrt{c}\(\hat{Q}_1,\hat{Q}_2,0,\cdots,0\).$$
Therefore, we have completed the proof of the equality condition in Theorem \ref{DDVVQ(p,q)}.

\subsection{Proof of Theorem \ref{DDVVQ(n-1,1)}}\label{sec2.2}
We recall that $d=\operatorname{dim}\operatorname{Span}\{B_1,\cdots,B_m\}$.
Since the inequality (\ref{ineqQ(n-1,1)}) is invariant under rotations on the matrix tuple $(B_1,\cdots,B_m)$, we may assume without loss of generality that $B_r=0$ for each $r>d$ and
$\<B_r, B_s\>=0$ for any $1\leq r\neq s\leq d$.
For each $1\leq r\leq d$, let $B_r=\begin{pmatrix}
0 & b_r  \\
b_r^t & 0 \\
\end{pmatrix}$,
and let $b_r=(b_r^1, \cdots, b_r^{n-1})^t$. 
Then by Lagrange's identity and the Cauchy-Schwarz inequality, we have
$$\begin{aligned}
\sum^m_{r,s=1}\|\[B_r,B_s\]\|^2&
=\sum^d_{r,s=1}\sum^{n-1}_{i,j=1}\(b_r^ib_s^j-b_s^ib_r^j\)^2
=2\sum^d_{r,s=1}\(\|b_r\|^2\|b_s\|^2-\<b_r, b_s\>^2\)\\
&=2\sum_{1\leq r\neq s\leq d}\|b_r\|^2\|b_s\|^2
=2\[\(\sum^d_{r=1}\|b_r\|^2\)^2-\sum^d_{r=1}\|b_r\|^4\]\\
&\leq2\(1-\frac{1}{d}\)\(\sum^d_{r=1}\|b_r\|^2\)^2
=\frac{d-1}{2d}\(\sum^m_{r=1}\|B_r\|^2\)^2,
\end{aligned}$$
and the equality holds if and only if $\|B_1\|=\cdots=\|B_d\|$, or equivalently,
there exists $P\in\operatorname{O}(n)$ such that $P^tB_rP=D_r$ for each $1\leq r\leq d$.
Thus the proof is complete.

\subsection{Normal scalar curvature inequality}
Suppose $M^n$ is an austere submanifold of a real space form $N^{n+m}(\kappa)$
with constant sectional curvature $\kappa$.
Let $R$ be the Riemannian curvature tensor of $M$, and let
$R^{\bot}$ be the curvature tensor of the normal connection.
For an arbitrary point $p\in M$,
let $\{e_1,\cdots,e_n\}$ and $\{\xi_1,\cdots,\xi_m\}$ be orthonormal bases of
$T_pM$ and $T^{\bot}_pM$, respectively.
Then the normalized scalar curvature is
$$\rho=\frac{2}{n(n-1)}\sum_{1=i<j}^n\langle R(e_i, e_j)e_j, e_i\rangle$$
and the normalized normal scalar curvature is
$$\rho^{\bot}=\frac{2}{n(n-1)}\Bigg(\sum_{1=i<j}^n\sum_{1=r<s}^m\langle R^{\bot}(e_i, e_j)\xi_r, \xi_s\rangle^2\Bigg)^{\frac12}.$$
For each $1\leq r\leq m$, let $A_r$ be the matrix corresponding to the second fundamental form along the direction $\xi_r$ with respect to the basis $\{e_1,\cdots,e_n\}$.
Then the Gauss equation implies that
\begin{equation}\label{DDVVRHS}
\kappa-\rho=\frac1{n(n-1)}\|\operatorname{II}\|^2=\frac1{n(n-1)}\sum^m_{r=1}\|A_r\|^2,
\end{equation}
and the Ricci equation implies that
\begin{equation}\label{DDVVLHS}
\rho^{\bot}=\frac1{n(n-1)}\Bigg(\sum^m_{r,s=1}\|[A_r,A_s]\|^2\Bigg)^{\frac12}.
\end{equation}
Note that $\operatorname{dim}|\operatorname{II}_p|
=\operatorname{dim}\operatorname{Span}\{A_1,\cdots,A_m\}$.
Therefore, it is easy to see that Theorem \ref{GeoDDVVQ(p,q)} and \ref{GeoDDVVQ(n-1,1)}
follow from Theorem \ref{DDVVQ(p,q)} and \ref{DDVVQ(n-1,1)}.

\section{Examples}\label{sec3}
In this section, we introduce more examples of austere submanifolds of type $\mathcal{Q}(p,q)$.
\subsection{Generalized helicoids}
The following example was constructed by Bryant\cite{B91}
in the classification of simple austere submanifolds.
For given integers $2\leq s<n$ and constants $\lambda_0\geq0$, $\lambda_1\geq\cdots\geq\lambda_s>0$,
let $f: \mathbb{R}^n\rightarrow\mathbb{R}^{n+s}$ be the smooth map given by
$$\begin{aligned}
f(x_0, x_1, \cdots, x_{n-1})=\big(\lambda_0x_0, x_1\cos(\lambda_1x_0), x_1\sin(\lambda_1x_0),\cdots, x_s\cos(\lambda_sx_0&), x_s\sin(\lambda_sx_0), \\
&x_{s+1}, \cdots, x_{n-1}\big).
\end{aligned}$$
Then $M(s, \lambda)=f(\mathbb{R}^n)$ is an $n$-dimensional austere submanifold
of type $\mathcal{Q}(n-1, 1)$.
(Note that $M(s, \lambda)$ is singular along the subset
$\left\{f(x_0, x_1, \cdots, x_{n-1}): x_1=\cdots=x_s=0\right\}$ if $\lambda_0=0$.)
One can verify that, if $\lambda_0>0$ and $\lambda_1=\cdots=\lambda_s$,
then the equality in the normal scalar curvature inequality
$$\rho^{\bot}\leq\sqrt{\frac{s-1}{2s}}(\kappa-\rho)$$
holds on the subset $\left\{f(x_0, x_1, \cdots, x_{n-1}): x_1=\cdots=x_s=0\right\}$.

\subsection{Austere embeddings of Grassmann manifolds
$\operatorname{G}_2(\mathbb{R}^4)$ and $\operatorname{G}_2(\mathbb{C}^4)$}
In \cite{GZ23}, the authors considered the orbits of
the conjugate action of $\operatorname{SO}(n)$ (resp. $\operatorname{U}(n)$)
on traceless, symmetric (resp. Hermitian) matrices with Frobenius norm 1.
Their calculations showed that a singular orbit of this $\operatorname{SO}(4)$-action
(resp. $\operatorname{U}(4)$-action)
is diffeomorphic to the Grassmann manifold $\operatorname{G}_2(\mathbb{R}^4)$
(resp. $\operatorname{G}_2(\mathbb{C}^4)$),
and is an austere submanifold of type $\mathcal{Q}(2,2)$ (resp. $\mathcal{Q}(4,4)$)
in the unit sphere $\mathbf{S}^8$ (resp. $\mathbf{S}^{14}$).
In fact, these orbits are given by the standard embeddings of flag manifolds \cite{BCO}.

\subsection{Some isoparametric focal submanifolds satisfying Condition A}
For an isoparametric hypersurface in a unit sphere with multiplicities $(m_1, m_2)$,
there exist two focal submanifolds $M_+$ and $M_-$ of codimension $m_1+1$ and $m_2+1$, respectively.
A point in a focal submanifold is said to satisfy Condition A
if the kernels of all shape operators coincide \cite{OT75}.
The Condition A plays an important role in the research of isoparametric hypersurfaces
in unit spheres with $4$ distinct principal curvatures.
In \cite{GTY20},
the subset $C_A$ of the points satisfying Condition A was determined for all the focal submanifolds.
Thus, we know that the focal submanifolds satisfying Condition A everywhere are
$M_+$ with multiplicities $(2, 2)$, $(1, k)$($k\in\mathbb{N}^+$) and
$M_-$ with multiplicities $(2, 1)$, $(6, 1)$, $(4, 3)$(definite).
We are going to find austere submanifolds of type $\mathcal{Q}(p,q)$ in these focal submanifolds.

Let $M_+$ (resp. $M_-$) be the focal submanifold of an isoparametric hypersurface
with $4$ distinct principal curvatures and codimension $m_1+1$ (resp. $m_2+1$) in a unit sphere.
Suppose $M_+$ (resp. $M_-$) satisfies Condition A everywhere.
Then for any point $p\in M_+$ (resp. $M_-$ with $(m_1, m_2)$ interchanged),
there exist an orthonormal basis $\{e_1,\cdots,e_{m_1+2m_2}\}$ of $T_pM_+$ and an orthonormal basis $\{\xi_0,\xi_1,\cdots,\xi_{m_1}\}$ of $T_p^{\bot}M_+$ such that
the matrices corresponding to the shape operators in directions $\xi_0,\xi_1,\cdots,\xi_{m_1}$
can be written as
$$A_0=\begin{pmatrix}
I_{m_2} & 0 & 0\\
0 & -I_{m_2} & 0\\
0 & 0 & 0\\
\end{pmatrix}\ \mbox{and}\
A_r=
\begin{pmatrix}
0 & a_r & 0\\
a_r^t & 0 & 0\\
0 & 0 & 0\\
\end{pmatrix}\ \mbox{for}\ 1\leq r\leq m_1,$$
where $a_r\in M(m_2, m_2)$.
In addition, the restrictions of these shape operators to $\operatorname{Span}\{e_1,\cdots,e_{2m_2}\}$
induce a Clifford system $\{P_0,P_1,\cdots,P_{m_1}\}$ on $\mathbb{R}^{2m_2}$ (cf. \cite{Chi12}).

The key observation is that $M_+$ is an austere submanifold of type $\mathcal{Q}(p,q)$
if and only if there exists $P\in\operatorname{O}(m_1+2m_2)$ such that
$$P^2=I_{m_1+2m_2},\ P\neq\pm I_{m_1+2m_2}\ \mbox{and}\ PA_r=-A_rP\ \mbox{for}\ 0\leq r\leq m_1,$$
or equivalently, there exists $P_{m_1+1}\in\operatorname{O}(2m_2)$ such that
$\{P_0,P_1,\cdots,P_{m_1}, P_{m_1+1}\}$ forms a Clifford system on $\mathbb{R}^{2m_2}$.
Since $m_1=1,2\not\equiv0\,(\operatorname{mod}4)$,
by the theory of Clifford systems \cite{FKM81},
every Clifford system $\{Q_0,Q_1,\cdots,Q_{m_1}\}$ on $\mathbb{R}^{2m_2}$
is algebraically equivalent to $\{P_0,P_1,\cdots,P_{m_1}\}$, i.e.,
there exists $U\in\operatorname{O}(2m_2)$ such that $P_r=UQ_rU^t$ for $0\leq r\leq m_1$.
Hence, the Clifford system $\{P_0,P_1,\cdots,P_{m_1}\}$ can be extended to
a Clifford system $\{P_0,P_1,\cdots,P_{m_1}, P_{m_1+1}\}$ on $\mathbb{R}^{2m_2}$ if and only if
$m_2=k\delta(m_1+1)$ for some positive integer $k$, where $\delta(m)$ denotes the dimension of irreducible representation of the Clifford algebra $\mathcal{C}_{m-1}$ which satisfies $\delta(m+8)=16\delta(m)$ and can be listed in the following Table \ref{table3}.
\begin{table}[h!]
\caption{Dimension of irreducible representation of the Clifford algebra $\mathcal{C}_{m-1}$}\label{table3}
\centering
\begin{tabular}{|c|c|c|c|c|c|c|c|c|c|}
\hline
$m$ & $1$ & $2$ & $3$ & $4$ & $5$ & $6$ & $7$ & $8$ & $\cdots~ m+8$\\
\hline
$\delta(m)$& $1$ &$2$ &$4$ &$4$ & $8$ & $8$ & $8$ & $8$ & $\cdots~ 16\delta(m)$\\
\hline
\end{tabular}
\end{table}

It follows that the focal submanifolds $M_+$ with multiplicities $(1, 2k)$
are austere submanifolds of type $\mathcal{Q}(m_2, m_1+m_2)$.
By interchanging the multiplicities $(m_1, m_2)$, we also obtain that
the focal submanifolds $M_-$ with multiplicities $(2, 1)$, $(6, 1)$, $(4, 3)$(definite)
are austere submanifolds of type $\mathcal{Q}(m_1, m_1+m_2)$.
However, the focal submanifolds $M_-$ with multiplicities $(2, 1)$, $(6, 1)$ are congruent to
focal submanifolds $M_+$ with multiplicities $(1, 2)$, $(1, 6)$.

Recalling the equality condition in Theorem \ref{DDVVQ(p,q)},
we find that $\{C_1,C_2\}$ forms a Clifford system on $\mathbb{R}^{4}$.
Note that
the Clifford system $\{C_1,C_2\}$ must be algebraically equivalent to the Clifford system $\{P_0,P_1\}$ given by the shape operators of the focal submanifold $M_+$ with multiplicities $(1, 2)$.
Then it follows from Theorem \ref{GeoDDVVQ(p,q)} that
the focal submanifold $M_+$ with multiplicities $(1, 2)$
achieves the equality in the normal scalar curvature inequality
$$\rho^{\bot}\leq\frac{\sqrt{2}}{2}(\kappa-\rho).$$

To sum up the above discussion, we have the following proposition.
\begin{prop}\label{example3}
Let $M_+$ $($resp. $M_-$$)$ be the focal submanifold of an isoparametric hypersurface
with $4$ distinct principal curvatures and codimension $m_1+1$ $($resp. $m_2+1$$)$ in a unit sphere.
\begin{itemize}
\item[(1)] For each positive integer $k$, the focal submanifold $M_+$ with multiplicities $(1, 2k)$
is an austere submanifold of type $\mathcal{Q}(2k, 2k+1)$.
\item[(2)] The focal submanifold $M_-$ with multiplicities $(4, 3)$$($definite$)$
is an austere submanifold of type $\mathcal{Q}(4, 7)$.
\item[(3)] The focal submanifold $M_+$ with multiplicities $(1, 2)$
attains the equality everywhere in the normal scalar curvature inequality $(\ref{GeoineqQ(p,q)})$ of Theorem $\ref{GeoDDVVQ(p,q)}$, i.e., $\rho^{\bot}=\frac{\sqrt{2}}{2}(\kappa-\rho)$.
\end{itemize}
\end{prop}

\section{BW-type Inequality and Simons-type Gap Theorem}\label{sec4}
\subsection{BW-type inequality}
\begin{lem}\label{BWQ(p,q)}
Suppose $B_1, B_2\in\mathcal{Q}(p,q)$. Then
$$\|[B_1,B_2]\|^2\leq\|B_1\|^2\|B_2\|^2.$$
The equality holds if and only if there exists $P\in\operatorname{O}(n)$ such that
$$(P^tB_1P, P^tB_2P)=(\lambda\operatorname{diag}(C_1,0),\mu\operatorname{diag}(C_2,0))$$
for some $\lambda,\mu\geq0$.
\end{lem}
\begin{proof}
Since the required inequality is also invariant under orthogonal congruences,
as in the proof of Lemma \ref{lem3},
we may assume without loss of generality that $p\geq q$ and
$$B_1=\begin{pmatrix}
0 & 0 &D \\
0 &0 &0 \\
D &0 &0
\end{pmatrix},\
B_2=
\begin{pmatrix}
0 &0 &b \\
0 &0 &c \\
b^t &c^t &0
\end{pmatrix},$$
where $b=(b_{ij})_{q\times q}$, $c=(c_{ij})_{(p-q)\times q}$ and
$D=\operatorname{diag}(\lambda_1, \cdots, \lambda_q)$ with $\lambda_1\geq\cdots\geq\lambda_q\geq0$.
Let $X=(x_{ij})_{q\times q}$ and $Y=(y_{ij})_{q\times q}$
denote the symmetric and skew-symmetric parts of $b$, respectively.
Then we also have
$$\|[B_1,B_2]\|^2=2\sum_{i,j=1}^q(\lambda_i-\lambda_j)^2x_{ij}^2+2\sum_{i,j=1}^q(\lambda_i+\lambda_j)^2y_{ij}^2
+2\sum_{i=1}^{p-q}\sum_{j=1}^{q}\lambda_j^2c_{ij}^2.$$
Since $(\lambda_i-\lambda_j)^2\leq \lambda_i^2+\lambda_j^2\leq \sum_{i=1}^q\lambda_i^2=\frac12\|B_1\|^2$ and $(\lambda_i+\lambda_j)^2\leq2(\lambda_i^2+\lambda_j^2)\leq2\sum_{i=1}^q\lambda_i^2=\|B_1\|^2$
for any $1\leq i\neq j\leq n$, we obtain
$$\|[B_1,B_2]\|^2
\leq\|B_1\|^2\(\|X\|^2+2\|Y\|^2+\|c\|^2\)
\leq\|B_1\|^2\|B_2\|^2.$$
The equality condition is easily seen from the proof.
\end{proof}

\begin{lem}\label{BWQ(p,q)2}
If $B_1, B_2, B_3\in\mathcal{Q}(p,q)$ such that
$$\|[B_r,B_s]\|^2=\|B_r\|^2\|B_s\|^2\ \mbox{for any}\ 1\leq r\neq s\leq 3,$$
then at least one of the matrices $B_1, B_2, B_3$ is zero.
\end{lem}
\begin{proof}
If $B_1, B_2, B_3$ are all nonzero matrices,
then we may assume without loss of generality that $\|B_1\|=\|B_2\|=\|B_3\|=1$.
Then we have $$\|[B_1, B_2]\|^2=\|[B_2, B_3]\|^2=\|[B_3, B_1]\|^2=1$$
and thus by Theorem \ref{DDVVQ(p,q)}, we obtain
$$6=\sum_{r, s=1}^3\|[B_r, B_s]\|^2\leq \frac{1}{2}\(\sum_{r=1}^3\|B_r\|^2\)^2=\frac{9}{2},$$
which is a contradiction.
\end{proof}

\subsection{Simons-type gap theorem}
Let $f: M^n\rightarrow N^{n+m}(\kappa)$ be an isometric immersion into a real space form,
and let $(E_1, \dots, E_{n+m})$ be a local orthonormal frame of $N^{n+m}(\kappa)$ such that
$E_1, \dots, E_n$ are tangent to $M$.
We use $\big\{\overline{\theta}_i: 1\leq i\leq n+m\big\}$ and
$\big\{\overline{\omega}_i^j: 1\leq i, j\leq n+m\big\}$ to denote the dual 1-forms and connection 1-forms
corresponding to $(E_1, \dots, E_{n+m})$, respectively.
Then the structure equations of $N^{n+m}(\kappa)$ are given by
$$\begin{cases}
&d\overline{\theta}_i=\sum\limits_{j=1}^{n+m}\overline{\omega}_i^j\wedge\overline{\theta}_j,\
\overline{\omega}_i^j=-\overline{\omega}_j^i,\\
&d\overline{\omega}_i^j=\sum_{k=1}^{n+m}\limits\overline{\omega}_i^k\wedge\overline{\omega}_k^j-\overline{R}_i^j,
\end{cases}$$
where $\overline{R}_i^j=\kappa\overline{\theta}_i\wedge\overline{\theta}_j$ denotes the curvature 2-forms
of the real space form $N^{n+m}(\kappa)$.
Let $\theta_i=f^*\overline{\theta}_i$, $\omega_i^j=f^*\overline{\omega}_i^j$
for all $1\leq i, j\leq n+m$. Then for $1\leq i, j\leq n$, we have the structure equations
$$\begin{cases}
&d\theta_i=\sum\limits_{j=1}^{n+m}\omega_i^j\wedge\theta_j,\ \omega_i^j=-\omega_j^i,\\
&d\omega_i^j=\sum_{k=1}^{n+m}\limits\omega_i^k\wedge\omega_k^j-R_i^j,
\end{cases}$$
where $R_i^j=f^*\overline{R}_i^j-\sum_{k=n+1}^{n+m}\limits\omega_i^k\wedge\omega_k^j$.
In addition, for each $1\leq r\leq m$, we have
$$\theta_{n+r}=0\ \mbox{and}\ \omega_i^{n+r}=\sum_{j=1}^nh_{ij}^r\theta_j,$$
where $(h_{ij}^r)_{n\times n}=A_r$ is the matrix corresponding to the second fundamental form along the direction $E_{n+r}$, i.e., the second fundamental form of $M$ can be written as
$$\operatorname{II}=\sum_{r=1}^m\sum_{i,j=1}^nh_{ij}^r\theta_i\otimes\theta_j\otimes E_{n+r}.$$
If we write $$\nabla\operatorname{II}=
\sum_{r=1}^m\sum_{i,j,k=1}^nh_{ijk}^r\theta_i\otimes\theta_j\otimes\theta_k\otimes E_{n+r},$$
then \begin{equation}\label{h_ijk}
\sum_{k=1}^nh^r_{ijk}\theta^k=dh^r_{ij}-\sum_{k=1}^nh^r_{kj}\omega_i^k
-\sum_{k=1}^nh^r_{ik}\omega_j^k+\sum_{s=1}^mh^s_{ij}\omega_{n+s}^{n+r}.
\end{equation}
Let $S=\|\operatorname{II}\|^2=\sum\limits_{r=1}^{m}\|A_r\|^2$
be the squared length of the second fundamental form.
We have the following Simons-type gap theorem for austere submanifolds of type $\mathcal{Q}(p,q)$.

\begin{thm}\label{Simons}
Suppose $M^n$ is a closed austere submanifold of type $\mathcal{Q}(p,q)$ in the unit sphere $\mathbf{S}^{n+m}$ and $m\geq1$.
Then $$\int_M (S-n)S\,dV_M\geq0.$$
In addition, if $S\leq n$, then $M$ is either a totally geodesic submanifold or
the Clifford torus $S^p\(\sqrt{\frac12}\)\times S^p\(\sqrt{\frac12}\)$ with $n=2p$.
\end{thm}
\begin{proof}
First we recall the well-known Simons-type formula
\begin{equation}\label{simons-for}
\frac12\Delta S=\|\nabla\operatorname{II}\|^2+nS
-\sum_{r,s=1}^m\|[A_r, A_s]\|^2-\sum_{r,s=1}^m|\<A_r, A_s\>|^2.
\end{equation}
For any point $p\in M$, by choosing suitable basis in normal direction,
we may assume that $\<A_r, A_s\>=0$ for any $1\leq r\neq s\leq m$.
It follows from Lemma \ref{BWQ(p,q)} that
$$\sum_{r,s=1}^m\|[A_r, A_s]\|^2\leq\sum_{1\leq r\neq s\leq m}\|A_r\|^2\|A_s\|^2
=\(\sum_{r=1}^m\|A_r\|^2\)^2-\sum_{r=1}^m\|A_r\|^4,$$
i.e., $$\sum_{r,s=1}^m\|[A_r, A_s]\|^2+\sum_{r=1}^m\|A_r\|^4\leq S^2.$$
Then by the Simons-type formula (\ref{simons-for}), we have
\begin{equation}\label{simons-for2}
\frac12\Delta S\geq\|\nabla\operatorname{II}\|^2+nS-S^2.
\end{equation}
Integrating (\ref{simons-for2}) over $M$, we obtain
$$\int_M (S-n)S\,dV_M\geq\int_M\|\nabla\operatorname{II}\|^2\,dV_M\geq0.$$
If in addition $S\leq n$,
then we have $\nabla\operatorname{II}=0$ and $S\equiv0$ or $n$.
When $S\equiv0$, $M$ is a totally geodesic submanifold.
When $S\equiv n$, we divide the discussion into two cases:

$(1)$ If $m=1$, then it follows from \cite[Theorem 2]{CDK} that $M$ is isometric to one of the Clifford torus $S^k\(\sqrt{\frac{k}{n}}\)\times S^{n-k}\(\sqrt{\frac{n-k}{n}}\)$, $1\leq k\leq n$.
However, only when $n=2k$ and $k=p=q$, the submanifold $M$ is of type $\mathcal{Q}(p,q)$.

$(2)$ If $m\geq2$, then by Lemma \ref{BWQ(p,q)2}, we may assume without loss of generality that
\begin{equation}\label{cond}
A_1=\lambda\operatorname{diag}(C_1,0), A_2=\mu\operatorname{diag}(C_2,0)
\ \mbox{and}\ A_r=0\ \mbox{for}\ r>2,
\end{equation}
where $\lambda,\mu$ are smooth function on $M$. In other words,
$$\omega^{n+1}_1=\lambda\theta_3,\ \omega^{n+1}_2=\lambda\theta_4,\ \omega^{n+1}_3=\lambda\theta_1,\ \omega^{n+1}_4=\lambda\theta_2,\ \omega^{n+1}_i=0\ \mbox{for}\ i>4,$$
$$\omega^{n+2}_1=-\mu\theta_4,\ \omega^{n+2}_2=\mu\theta_3,\ \omega^{n+2}_3=\mu\theta_2,\ \omega^{n+2}_4=-\mu\theta_1,\ \omega^{n+2}_i=0\ \mbox{for}\ i>4,$$
$$\omega^{n+r}_i=0\ \mbox{for}\ r>2\ \mbox{and}\ 1\leq i\leq n.$$
Since $\nabla\operatorname{II}=0$, equation (\ref{h_ijk}) becomes
\begin{equation}\label{dh_i}
dh^r_{ij}=\sum_{k=1}^nh^r_{kj}\omega_i^k+\sum_{k=1}^nh^r_{ik}\omega_j^k-\sum_{s=1}^mh^s_{ij}\omega_{n+s}^{n+r}.
\end{equation}
Setting $r=1$, $j\geq5$ and $i=1, 2, 3, 4$, we obtain $\omega^3_j=\omega^4_j=\omega^1_j=\omega^2_j=0$, respectively.
It follows that for each $j\geq5$,
$$0=d\omega^1_j=\sum_{k=1}^n\omega^k_j\wedge\omega^1_k-R^1_j=-\theta_1\wedge\theta_j.$$
Since $\theta_1, \cdots, \theta_n$ are orthonormal,
we have $\theta_j=0$ for $j\geq 5$, and thus $\operatorname{dim}M=4$.
Direct calculations show that for any $a,b\in\mathbb{R}$,
the eigenvalues of $aC_1+bC_2$ are $\pm\sqrt{a^2+b^2}$ of multiplicity 2.
In other words, $M$ has at most two distinct principal curvatures of multiplicity $2$ in every normal direction. Then it follows from \cite[Theorem 1.5]{JM84} that
$M$ is given by the Veronese embeddings of complex projective plane.
But the matrices corresponding to the second fundamental form of this Veronese embedding does not satisfy condition (\ref{cond}). (One can find an explicit description in \cite[p. 560]{Li16}.)
\end{proof}



\end{document}